\documentclass[11pt]{amsart}

\usepackage{amssymb,amsfonts,mathrsfs,enumitem,amsmath,amsthm,bbold}
\usepackage[mathscr]{euscript}
\usepackage{color}

\usepackage{tikz,tikz-cd}
\usetikzlibrary{angles,quotes,calc,arrows.meta,decorations.markings,ext.arrows}

\usepackage[backend=bibtex,
    sorting=anyt,
    isbn=false,
    url=false,
    doi=false,
    maxbibnames=100
    ]{biblatex}
\definecolor{Chocolat}{rgb}{0.36, 0.2, 0.09}
\definecolor{BleuTresFonce}{rgb}{0.215, 0.215, 0.36}
\definecolor{EgyptianBlue}{rgb}{0.06, 0.2, 0.65}

\usepackage[colorlinks,final,hyperindex]{hyperref}
\hypersetup{citecolor=BleuTresFonce, urlcolor=EgyptianBlue, linkcolor=Chocolat}

\usepackage{fourier}

\newtheorem{theorem}{Theorem}[section]

\newtheorem{proposition}[theorem]{Proposition}

\theoremstyle{definition}

\DeclareMathOperator{\Lie}{Lie}
\DeclareMathOperator{\Aut}{Aut}

\DeclareMathOperator{\Mod}{Mod}

\DeclareMathOperator{\End}{End}

\DeclareMathOperator{\tr}{tr}

\DeclareMathAlphabet{\mathbbold}{U}{bbold}{m}{n}

\def\k{\mathbbold{k}}

\begin{document}

\title[Algebraic versions of $\mathbb{T}^2$ and of $\mathbb{P}^1\times\mathbb{P}^1$ and Hochschild cohomology]{Algebraic versions of $\mathbb{T}^2$ and of $\mathbb{P}^1\times\mathbb{P}^1$\\  and Hochschild cohomology}

\author{Vladimir Dotsenko}

\address{Institut de Recherche Math\'ematique Avanc\'ee, UMR 7501, Universit\'e de Strasbourg et CNRS, 7 rue Ren\'e-Descartes, 67000 Strasbourg CEDEX, France}

\email{vdotsenko@unistra.fr}

\author{Andrea Solotar}

\address{IMAS-CONICET and Departamento de Matemática, Facultad de Ciencias Exactas y Naturales, Universidad de Buenos Aires, Pabellon I, Ciudad Universitaria, Buenos Aires, 1428, Argentina}

\address{Guangdong Technion Israel Institute of Technology, Shantou, Guangdong Province, China}

\email{asolotar@dm.uba.ar}

\date{}

\dedicatory{To the memory of Roberto Martínez-Villa}

\begin{abstract}
We examine the Hochschild cohomology for triangular algebras that capture some aspects of geometry and topology of the torus and of the quadric surface, and for deformations of these algebras. In particular, this shows that the cup product on the Hochschild cohomology of a triangular algebra does not generally follow the intuition coming from monomial algebras. Our examples also demonstrate that the Hochschild cohomology of a deformation of an algebra may not experience the dimension drop but still have a different cup product structure, and that the Hochschild cohomologies of deformations of two derived equivalent algebras may exhibit noticeably different behaviours. 
\end{abstract}

\maketitle

\section{Introduction}

For an associative algebra $\Lambda$, one of its most fundamental invariants is the Hochschild cohomology $HH^\bullet(\Lambda,\Lambda)$ \cite{MR11076}, equipped with its Gerstenhaber algebra structure \cite{MR161898}; in fact, this data is, in a suitable sense, invariant under derived equivalences \cite{MR2043327,MR1099084}. It occasionally happens that the derived category of $\Lambda$-modules and the derived category of coherent sheaves on a smooth scheme~$X$ are equivalent; such derived equivalences tend to arise from exceptional collections \cite{MR928291,MR992977}, which forces $\Lambda$ to be \emph{triangular}, so that $\Lambda=\k{}Q/(R)$ is a quotient of the path algebra of a quiver $Q$ without directed cycles.

Gröbner bases methods~\cite{MR1227656} allow one to view any such algebra as a filtered ``deformation'' of a \emph{monomial} algebra, that is of an algebra whose relations are all paths, and not nontrivial linear combinations of paths. In \cite{MR2185621}, Bustamante conjectured that if $\Lambda$ is a monomial triangular algebra, then the cup product of $HH^\bullet(\Lambda,\Lambda)$ vanishes on elements of positive degree. This conjecture was recently proved by Artenstein, Letz, Oswald and the second author \cite{MR4750103}. Note that from the point of view of derived categories, monomial algebras are known to occasionally exhibit somewhat pathological behaviour: the category of perfect complexes of modules may not have exceptional objects \cite{MR1116189}, or may have nonextendable exceptional objects \cite{kuznetsov2013simplecounterexamplejordanholderproperty}.

The usual semi-continuity arguments \cite[Sec.~III.12]{MR0463157} imply that the ranks of Hochschild cohomology groups of an algebra $\Lambda$ do not exceed the corresponding ranks for the monomial algebra whose relations are the leading terms of a Gröbner basis of $\Lambda$. A somewhat naïve first reflex would be to suggest that some sort of semi-continuity also holds for the \emph{structure}, and to extend the conjecture of Bustamante by suggesting that the vanishing of the cup product holds for more general triangular algebras. In this short note, we discuss several examples of triangular algebras that show this expectation to be extremely far from true. Our examples are of two flavours: incidence algebras of two posets whose geometric realizations is the torus $\mathbb{T}^2$, treated in Section \ref{sec:incidence}, and two algebras arising from exceptional collections in categories of coherent sheaves on $\mathbb{P}^1\times\mathbb{P}^1$, treated in Section \ref{sec:exceptional}. One of the two latter algebras is merely the tensor square of the path algebra of the Kronecker quiver; in fact, already Green and Solberg~\cite{MR2067380} indicate that tensor products of algebras may be used to provide examples of algebras with nonzero cup products of Hochschild cohomology. 
Most of the algebras $\Lambda$ we encounter are quadratic Koszul algebras of global dimension two, which suggests that, beyond the case of monomial triangular algebras, there is unlikely to be any larger natural class of algebras for which the nontrivial cup products vanish. 

Besides indicating that examples of nonzero cup products in Hochschild cohomology of triangular algebras are abundant, once one looks at algebras of geometric and topological origin, we investigate the Hochschild cohomology of deformations of our algebras, describing how the dimension of the Hochschild cohomology depends on the deformation parameters; the corresponding results are presented in Sections \ref{sec:defS1S1} and \ref{sec:defP1P1}. Interestingly enough, even where is no dimension drop for Hochschild cohomology of the deformed algebra, most nonzero cup products of elements of positive degrees seem to ``disappear'' after deformation. Our computations of Hochschild cohomology for deformations of derived $\mathbb{P}^1\times\mathbb{P}^1$, presented in Section \ref{sec:defP1P1}, exhibit another interesting feature: they are qualitatively different from those of Belmans \cite{MR3988086}, suggesting that while derived equivalent algebras have isomorphic Hochschild cohomology, the cohomological behaviour of their deformations may differ drastically.

\subsection*{Conventions. }All vector spaces in this note are defined over a ground field $\k$; when assumptions on its characteristic are necessary, we give them explicitly.
We refer the reader to the textbook of Witherspoon \cite{MR3971234} for all the relevant definitions concerning the Hochschild cohomology of associative algebras, and to the survey of Martínez-Villa \cite{MR2388890} and the article of Chouhy and the second author \cite{MR3334140} for notations and conventions concerning quotients of path algebras and their cohomological invariants. The \texttt{Magma} \cite{MR1484478} and \texttt{QPA} \cite{qpa} code used to check some of the claims of the paper is available upon request.

\subsection*{Acknowledgements}

This research was supported by the Math-AmSud project HHMA (23-MATH-06) and by Institut Universitaire de France. We acknowledge interesting conversations with Pedro Tamaroff. Special thanks are due to Maxim Smirnov for discussions of exceptional collections in triangulated categories and to \O{}yvind Solberg for patient explanations of the \texttt{QPA} package \cite{qpa}. We are also indebted to Pieter Belmans, Claude Cibils and Nikita Markarian for useful comments on a draft version of the paper. 

\section{Two algebraic versions of \texorpdfstring{$\mathbb{T}^2$}{T2}}\label{sec:incidence}

In this section, we examine two examples of nonzero cup products arising from two different algebraic versions for $\mathbb{T}^2$, by which we mean incidence algebras $I(P)$ of partially ordered sets $P$ whose geometric realization is the torus $\mathbb{T}^2$. As shown by Gerstenhaber and Schack \cite{MR722369}, for a poset $P$, the Hochschild cohomology of the incidence algebra $I(P)$ is isomorphic to the simplicial cohomology of the geometric realization of $P$; crucially for us, this isomorphism is a Gerstenhaber algebra isomorphism (more precisely, it is compatible with cup products, and the Gerstenhaber bracket vanishes). Thus, if the geometric realization of $P$ is the torus, then we have a Gerstenhaber algebra isomorphism
 \[
HH^\bullet(I(P),I(P))\cong H^\bullet(\mathbb{T}^2)\cong H^\bullet(S^1)^{\otimes 2}.   
 \] 
In particular, the Hochschild cohomology of such an algebra always has a nonzero cup product; of course, incidence algebras are always triangular. Exceptional collections for incidence algebras are discussed in \cite{MR2357344}; we are not going to use them in this paper. 

\subsection{Two posets with the homotopy type of a torus}

The minimal simplicial torus~$\mathbb{T}^2_s$ is displayed\footnote{We learned this particularly beautiful way of representing it from the notes of Boardman \cite{torus}.} in Figure \ref{fig1}; it is essentially a distorted rectangle drawn on the grid of regular triangles (the vertex labels and the arrows on the opposite sides should help the reader to unambiguously glue it into a torus). 

\begin{figure}[h]
  \centering
\begin{tikzpicture}[scale=1, every node/.style={font=\sffamily},
    vert/.style={inner sep=0pt, minimum size=0pt},
    triedge/.style={line width=0.8pt},
    idarrow/.style={-{Stealth[length=6pt,width=6pt]},semithick},
    doublehead/.style={
    semithick,
    postaction={
      decorate,
      decoration={
        markings,
        mark=at position 0.88 with {\arrow{Stealth[length=6pt,width=6pt]}},
        mark=at position 0.96 with {\arrow{Stealth[length=6pt,width=6pt]}}
      }}}  ]

\def\side{1.4}                    
\pgfmathsetmacro{\hs}{sqrt(3)/2*\side}  

\node[vert] (Atop) at (0*\side,4*\hs) {$0$};
\node[vert] (Btop) at (1*\side,4*\hs) {$1$};
\node[vert] (Ctop) at (2*\side,4*\hs) {$2$};
\node[vert] (F13) at (0.5*\side,3*\hs) {$3$};
\node[vert] (G13) at (1.5*\side,3*\hs) {$5$};
\node[vert] (A13) at (2.5*\side,3*\hs) {$0$};
\node[vert] (D2) at (1*\side,2*\hs) {$4$};
\node[vert] (E2) at (2*\side,2*\hs) {$6$};
\node[vert] (F2) at (3*\side,2*\hs) {$3$};
\node[vert] (A10) at (0.5*\side,1*\hs) {$0$};
\node[vert] (B10) at (1.5*\side,1*\hs) {$1$};
\node[vert] (C10) at (2.5*\side,1*\hs) {$2$};
\node[vert] (D10) at (3.5*\side,1*\hs) {$4$};
\node[vert] (Abot) at (3*\side,0*\hs) {$0$};

\foreach \u/\v in {
  Atop/Btop, Btop/Ctop,
  Atop/F13, Btop/F13, Btop/G13, Ctop/G13, Ctop/A13,
  F13/G13, G13/A13, F13/D2, G13/D2, G13/E2, A13/E2, A13/F2,
  D2/E2, E2/F2,
  D2/A10, D2/B10, E2/B10, E2/C10, F2/C10, F2/D10,
  A10/B10, B10/C10, C10/D10,
  C10/Abot, D10/Abot} {
  \draw[triedge] (\u) -- (\v);
}

\draw[idarrow] (Atop) -- (Btop);
\draw[idarrow] (Btop) -- (Ctop);
\draw[idarrow] (Ctop) -- (A13);
\draw[idarrow] (A10) -- (B10);
\draw[idarrow] (B10) -- (C10);
\draw[idarrow] (C10) -- (Abot);

\draw[doublehead] (Atop) -- (F13);
\draw[doublehead] (F13) -- (D2);
\draw[doublehead] (D2) -- (A10);
\draw[doublehead] (A13) -- (F2);
\draw[doublehead] (F2) -- (D10);
\draw[doublehead] (D10) -- (Abot);
\end{tikzpicture}  
  \caption{The minimal simplicial torus~$\mathbb{T}^2_s$.}\label{fig1}
\end{figure}
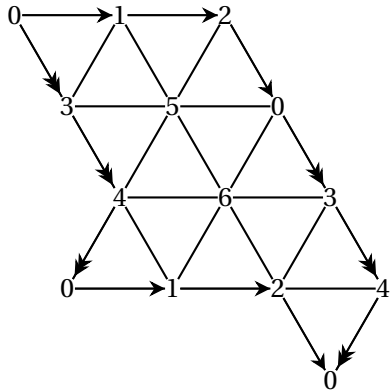

This simplicial complex has seven vertices $\{0,1,\ldots,6\}$, all $21$ edges between them are $1$-faces, and there are $14$ faces of dimension two. The incidence algebra $I(\mathbb{T}^2_s)$ is thus the quotient of the path algebra of the quiver $Q(\mathbb{T}^2_s)$ on $42$ vertices with arrows $\alpha^a_{a,b}$ that correspond to embedding $\{a\}\subset\{a,b\}$ of a vertex into a $1$-face and arrows $\beta^{a,b}_{a,b,c}$ of $1$-faces into $2$-faces. Relations of this particular incidence algebra are all quadratic; a typical relation corresponds to a vertex $a$ of a $2$-face $\{a,b,c\}$, and is of the form
 \[
\beta^{a,b}_{a,b,c}\alpha^a_{a,b}=\beta^{a,c}_{a,b,c} \alpha^a_{a,c} . 
 \]  

The minimal cubical torus~$\mathbb{T}^2_c$ is displayed in Figure \ref{fig2} (the vertex labels and the arrows on the opposite sides should help the reader to unambiguously glue it into a torus).
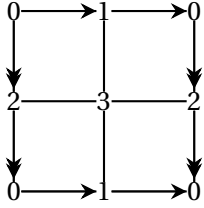
\begin{figure}[h]
  \centering

\begin{tikzpicture}[scale=1, every node/.style={font=\sffamily},
    vert/.style={inner sep=0pt, minimum size=0pt},
    triedge/.style={line width=0.8pt},
    idarrow/.style={-{Stealth[length=6pt,width=6pt]},semithick},
    doublehead/.style={
    semithick,
    postaction={
      decorate,
      decoration={
        markings,
        mark=at position 0.88 with {\arrow{Stealth[length=6pt,width=6pt]}},
        mark=at position 0.96 with {\arrow{Stealth[length=6pt,width=6pt]}}
      }}}  ]

\def\side{1.2}                    
\pgfmathsetmacro{\hs}{\side}  

\node[vert] (A2) at (0*\side,2*\hs) {$0$};
\node[vert] (B2) at (1*\side,2*\hs) {$1$};
\node[vert] (C2) at (2*\side,2*\hs) {$0$};
\node[vert] (A1) at (0*\side,1*\hs) {$2$};
\node[vert] (B1) at (1*\side,1*\hs) {$3$};
\node[vert] (C1) at (2*\side,1*\hs) {$2$};
\node[vert] (A0) at (0*\side,0*\hs) {$0$};
\node[vert] (B0) at (1*\side,0*\hs) {$1$};
\node[vert] (C0) at (2*\side,0*\hs) {$0$};

\foreach \u/\v in {
  A2/B2, B2/C2,
  A2/A1, A1/B1, B1/C1, B2/B1, C2/C1,
  A1/A0, A0/B0, B0/C0, B1/B0, C1/C0} {
  \draw[triedge] (\u) -- (\v);
}

\draw[idarrow] (A2) -- (B2);
\draw[idarrow] (B2) -- (C2);
\draw[idarrow] (A0) -- (B0);
\draw[idarrow] (B0) -- (C0);

\draw[doublehead] (A2) -- (A1);
\draw[doublehead] (A1) -- (A0);
\draw[doublehead] (C2) -- (C1);
\draw[doublehead] (C1) -- (C0);
\end{tikzpicture}  
 \caption{The minimal cubical torus~$\mathbb{T}^2_c$.}\label{fig2}
\end{figure}  

This simplicial complex has four $0$-faces $\{0,1,2,3\}$, eight edges between them as $1$-faces, and four faces of dimension two. The incidence algebra $I(\mathbb{T}^2_c)$ is thus the quotient of the path algebra of the quiver $Q(\mathbb{T}^2_c)$ on $16$ vertices with arrows $\alpha^a_{a,b}$ that correspond to embedding $\{a\}\subset\{a,b\}$ of a $0$-face into a $1$-face and arrows $\beta^{a,b}_{a,b,c,d}$ of $1$-faces into $2$-faces. Relations of this particular incidence algebra are all quadratic; a typical relation corresponds to a $0$-face $\{a\}$ of a $2$-face $\{a,b,c,d\}$, and is of the form
 \[
\beta^{a,b}_{a,b,c,d}\alpha^a_{a,b}=\beta^{a,d}_{a,b,c,d} \alpha^a_{a,d} . 
 \]  

Note that in both cases each relation of the incidence algebra corresponds to a $2$-face and one of its $0$-faces, and has two terms that correspond to the $1$-faces of the given $2$-face that contain the given $0$-face; it will be convenient to direct the relation so that the direction from the $1$-face corresponding to the leading term to the $1$-face corresponding to the irreducible path is positive (counterclockwise) for the standard orientation of the torus. 

From the main result of \cite{MR722369}, we immediately derive the following statement.

\begin{proposition}
We have isomorphisms of graded commutative algebras
 \[
HH^\bullet(I(\mathbb{T}^2_s),I(\mathbb{T}^2_s))\cong H^\bullet(\mathbb{T}^2,\k)\cong HH^\bullet(I(\mathbb{T}^2_c),I(\mathbb{T}^2_c)).     
 \]
In particular, the cup product of the Hochschild cohomology of the incidence algebras $I(\mathbb{T}^2_s)$ and $I(\mathbb{T}^2_c)$ is nonzero.
\end{proposition}

While the torus gives one of the smallest possible examples with a nonzero cup product in the cohomology, one can build an incidence algebra for any triangulation of a topological space with nonzero products in cohomology, e.g., $\mathbb{C}P^n$ for $n>1$. However, an additional nice feature of our example is that all paths in $\k Q(\mathbb{T}^2_s)$ and $\k Q(\mathbb{T}^2_c)$ are of length at most two, and therefore the algebras $I(\mathbb{T}^2_s)$ and $I(\mathbb{T}^2_c)$ possess some additional agreeable properties: they are Koszul and of global dimension two.  

In the next section, we shall examine Hochschild cohomology of deformations of the algebras $I(\mathbb{T}^2_s)$ and $I(\mathbb{T}^2_c)$. Let us record the main result that we shall require.

\begin{proposition}\label{prop:HHincid}
Each of the incidence algebras $I(\mathbb{T}^2_s)$ and $I(\mathbb{T}^2_c)$ admits, up to equivalence, a one-parametric family of deformations, defined over $\k[\hbar]\subset\k[[\hbar]]$. If we denote these families by $I_q(\mathbb{T}^2_s)$ and $I_q(\mathbb{T}^2_c)$ respectively, with $q=1+\hbar$, then 
\begin{enumerate}
\item the algebra $I_q(\mathbb{T}^2_s)$ has defining relations corresponding to inclusions of a $0$-face $\{a\}$ into a $2$-face $\{a,b,c\}$ of $\mathbb{T}^2_s$; each such relation is of the form
 \[
\beta^{a,b}_{a,b,c}\alpha^a_{a,b}=q\beta^{a,c}_{a,b,c} \alpha^a_{a,c} ,
 \]  
where the direction from $\{a,b\}$ to $\{a,c\}$ is positive for the standard orientation of the torus. 

\item the algebra $I_q(\mathbb{T}^2_c)$ has defining relations corresponding to inclusions of a $0$-face $\{a\}$ into a $2$-face $\{a,b,c,d\}$ of $\mathbb{T}^2_c$; each such relation is of the form
 \[
\beta^{a,b}_{a,b,c,d}\alpha^a_{a,b}=q\beta^{a,d}_{a,b,c,d} \alpha^a_{a,d} , 
 \]  
where the direction from $\{a,b\}$ to $\{a,d\}$ is positive for the standard orientation of the torus.
\end{enumerate}
\end{proposition}

\begin{proof}
Note that $H^2(\mathbb{T}^2,\mathbb{Z})\cong\mathbb{Z}$, and therefore 
 \[
HH^2(I(\mathbb{T}^2_s),I(\mathbb{T}^2_s))\cong \k \cong HH^2(I(\mathbb{T}^2_c),I(\mathbb{T}^2_c)),       
 \]
meaning that, up to equivalence, each of the incidence algebras $I(\mathbb{T}^2_s)$ and $I(\mathbb{T}^2_c)$ has exactly one direction of infinitesimal deformations. To describe these infinitesimal deformations, we would like to know the non-zero classes of $HH^2$ explicitly, which requires a short computation. Let us explain it in detail in the case of $I(\mathbb{T}^2_s)$; the case of $I(\mathbb{T}^2_c)$ is completely analogous. 

In this case, one may use the Koszul complex \cite[Chapter~3]{MR3971234}, or the resolution of Chouhy and the second author \cite{MR3334140}; since the quiver $Q(\mathbb{T}^2_s)$ has no path of length three, it is obvious that the relations of $I(\mathbb{T}^2_s)$ form a Gröbner basis, and the two resolutions are easily seen to be isomorphic, giving the following complex computing the Hochschild cohomology:
 \[
0\to \k Q(\mathbb{T}^2_s)_0||Q(\mathbb{T}^2_s)_0 \to \k Q(\mathbb{T}^2_s)_1||Q(\mathbb{T}^2_s)_1\to \k L_2^s||N_2^s\to 0   , 
 \]
where $L_2^s=\{m^a_{a,b,c}\}$ is the set of leading terms of relations, and $N_2^s=\{n^a_{a,b,c}\}$ is the set of irreducible paths; note that for any pair of vertices of $Q(\mathbb{T}^2_s)$ corresponding to a $0$-face incident to a $2$-face, there are exactly two paths of length two between them, one of which is the leading term of the relation and the other is an irreducible path. Moreover, for the rightmost differential of that complex, we have
\begin{gather*}
\partial(\alpha^a_{a,b}||\alpha^a_{a,b})=m^a_{a,b,c}||n^a_{a,b,c}-m^a_{a,b,d}||n^a_{a,b,d},\\ 
\partial(\beta^{a,b}_{a,b,c}||\beta^{a,b}_{a,b,c})= m^a_{a,b,c}||n^a_{a,b,c}-m^b_{a,b,c}||n^b_{a,b,c},
\end{gather*}    
where $\{a,b,c\}$ and $\{a,b,d\}$ are the two $2$-faces containing the $1$-face $\{a,b\}$, so that the direction from $\{a,b\}$ to $\{a,c\}$ is positive, and the direction from $\{a,b\}$ to $\{a,d\}$ is negative. This implies that modulo the image of the differential all the elements $m^a_{a,b,c}||n^a_{a,b,c}$ are equal to each other, meaning that, up to equivalence, the nontrivial infinitesimal deformation is defined by the relations
 \[
\beta^{a,b}_{a,b,c}\alpha^a_{a,b}-\beta^{a,c}_{a,b,c} \alpha^a_{a,c}=\hbar \beta^{a,c}_{a,b,c} \alpha^a_{a,c},      
 \]
where the direction from $\{a,b\}$ to $\{a,c\}$ is positive. If we denote $1+\hbar=:q$, we obtain exactly the relations indicated above. 

It remains to use the fact that the quiver $Q(\mathbb{T}^2_s)$ has no path of length three once again: it immediately follows that our infinitesimal deformation is defined over $\k[\hbar]\subset\k[[\hbar]]$, and not just over $\k[h]/(\hbar^2)$, as an infinitesimal deformation usually would be. Indeed, our relations take care of products of two arrows, and powers of $\hbar$ would only appear when computing products of at least three arrows. This completes the proof.  
\end{proof}

\subsection{Hochschild cohomology of deformations of the incidence algebras}\label{sec:defS1S1}

Deformations of incidence algebras have been recently studied by Iovanov and Koffi \cite{MR4388790}; the arXiv version of \emph{op. cit.} raises a question of computing the Hoch\-schild cohomology of those deformations and conjectures that the result will be the same as in the undeformed case. Let us show that this is generally not the case.  

\begin{theorem} For the Hochschild cohomology of the deformed algebras $I_q(\mathbb{T}^2_s)$ and $I_q(\mathbb{T}^2_c)$,
\begin{enumerate}
\item we have
 \[
HH^\bullet(I_q(\mathbb{T}^2_s),I_q(\mathbb{T}^2_s))\cong
 \begin{cases}
H^\bullet(\mathbb{T}^2)\quad  \text{ if } q^3=1,\\
H^\bullet(S^1) \quad \text{ otherwise. }
 \end{cases} 
 \]
\item  
we have
 \[
HH^\bullet(I_q(\mathbb{T}^2_c),I_q(\mathbb{T}^2_c))\cong
 \begin{cases}
H^\bullet(\mathbb{T}^2)\quad  \text{ if } q^4=1,\\
H^\bullet(S^1) \quad \text{ otherwise. }
 \end{cases} 
 \]
\end{enumerate} 
\end{theorem} 

\begin{proof}
Once again, we shall consider the case of the algebra $I_q(\mathbb{T}^2_s)$ in detail, since the case of $I_q(\mathbb{T}^2_c)$ is completely analogous (and somewhat simpler), and use the same strategy as in Proposition \ref{prop:HHincid} for computing the Hochschild cohomology, that is via the complex
 \[
0\to \k Q(\mathbb{T}^2_s)_0||Q(\mathbb{T}^2_s)_0 \to \k Q(\mathbb{T}^2_s)_1||Q(\mathbb{T}^2_s)_1\to \k L_2^s||N_2^s\to 0   , 
 \]
where $L_2^s=\{m^a_{a,b,c}\}$ is the set of leading terms of relations, and $N_2^s=\{n^a_{a,b,c}\}$ is the set of irreducible paths of length two. The rightmost differential of that complex is given by
\begin{gather}
\partial(\alpha^a_{a,b}||\alpha^a_{a,b})=m^a_{a,b,c}||n^a_{a,b,c}-qm^a_{a,b,d}||n^a_{a,b,d},\label{eq:ar1}\\ 
\partial(\beta^{a,b}_{a,b,c}||\beta^{a,b}_{a,b,c})= m^a_{a,b,c}||n^a_{a,b,c}-qm^b_{a,b,c}||n^b_{a,b,c},\label{eq:ar2}
\end{gather}    
where $\{a,b,c\}$ and $\{a,b,d\}$ are the two $2$-faces containing the $1$-face $\{a,b\}$, so that the direction from $\{a,b\}$ to $\{a,c\}$ is positive, and the direction from $\{a,b\}$ to $\{a,d\}$ is negative.

If we sum Relations \eqref{eq:ar2} with coefficients $1,q,q^2$ over the three $1$-faces of a $2$-face $\{a,b,c\}$, we obtain $(1-q^3)m^a_{a,b,c}||n^a_{a,b,c}$, which shows that for $q^3\ne 1$ the differential $\partial$ is surjective, so $HH^2(I_q(\mathbb{T}^2_s),I_q(\mathbb{T}^2_s))=0$. Moreover, the center of the algebra $I_q(\mathbb{T}^2_s)$ is manifestly trivial, so $HH^0(I_q(\mathbb{T}^2_s),I_q(\mathbb{T}^2_s))=\k$. Computing the Euler characteristic of the complex above, we obtain
 \[
42-(7\cdot 6+14\cdot 2)+14\cdot 2=0,     
 \]
which implies that $\dim HH^1(I_q(\mathbb{T}^2_s),I_q(\mathbb{T}^2_s))=1$, proving that for $q^3\ne 1$, we have $HH^\bullet(I_q(\mathbb{T}^2_s),I_q(\mathbb{T}^2_s))\cong H^\bullet(S^1)$.

For $q^3=1$, the above formulas for the differential immediately show that modulo its image all elements $m^a_{a,b,c}||n^a_{a,b,c}$ are proportional, which implies that $\dim HH^2(I_q(\mathbb{T}^2_s),I_q(\mathbb{T}^2_s))\le 1$. Let us show that $HH^2(I_q(\mathbb{T}^2_s),I_q(\mathbb{T}^2_s))\ne 0$. Note that one can label the internal angles of all triangles of the triangular grid by the letters $x,y,z$ so that the six angles around each vertex have, in the counterclockwise order, the labels $x,y,z,x,y,z$, and the three angles of each face have, in the clockwise order, the labels $x,y,z$. Restricting this to $\mathbb{T}^2_s$, we obtain a labelling with the same property (one of such labellings is displayed in Figure \ref{fig3}).
\begin{figure}[h]
  \centering
\begin{tikzpicture}[scale=1, every node/.style={font=\sffamily},
    vert/.style={inner sep=0pt, minimum size=0pt},
    triedge/.style={line width=0.8pt},
    idarrow/.style={-{Stealth[length=6pt,width=6pt]},semithick},
    doublehead/.style={
    semithick,
    postaction={
      decorate,
      decoration={
        markings,
        mark=at position 0.88 with {\arrow{Stealth[length=6pt,width=6pt]}},
        mark=at position 0.96 with {\arrow{Stealth[length=6pt,width=6pt]}}
      }}}  ]

\def\side{1.4}                    
\pgfmathsetmacro{\hs}{sqrt(3)/2*\side}  

\node[vert] (Atop) at (0*\side,4*\hs) {$0$};
\node[vert] (Btop) at (1*\side,4*\hs) {$1$};
\node[vert] (Ctop) at (2*\side,4*\hs) {$2$};
\node[vert] (F13) at (0.5*\side,3*\hs) {$3$};
\node[vert] (G13) at (1.5*\side,3*\hs) {$5$};
\node[vert] (A13) at (2.5*\side,3*\hs) {$0$};
\node[vert] (D2) at (1*\side,2*\hs) {$4$};
\node[vert] (E2) at (2*\side,2*\hs) {$6$};
\node[vert] (F2) at (3*\side,2*\hs) {$3$};
\node[vert] (A10) at (0.5*\side,1*\hs) {$0$};
\node[vert] (B10) at (1.5*\side,1*\hs) {$1$};
\node[vert] (C10) at (2.5*\side,1*\hs) {$2$};
\node[vert] (D10) at (3.5*\side,1*\hs) {$4$};
\node[vert] (Abot) at (3*\side,0*\hs) {$0$};

\foreach \u/\v in {
  Atop/Btop, Btop/Ctop,
  Atop/F13, Btop/F13, Btop/G13, Ctop/G13, Ctop/A13,
  F13/G13, G13/A13, F13/D2, G13/D2, G13/E2, A13/E2, A13/F2,
  D2/E2, E2/F2,
  D2/A10, D2/B10, E2/B10, E2/C10, F2/C10, F2/D10,
  A10/B10, B10/C10, C10/D10,
  C10/Abot, D10/Abot} {
  \draw[triedge] (\u) -- (\v);
}

\draw[idarrow] (Atop) -- (Btop);
\draw[idarrow] (Btop) -- (Ctop);
\draw[idarrow] (Ctop) -- (A13);
\draw[idarrow] (A10) -- (B10);
\draw[idarrow] (B10) -- (C10);
\draw[idarrow] (C10) -- (Abot);

\draw[doublehead] (Atop) -- (F13);
\draw[doublehead] (F13) -- (D2);
\draw[doublehead] (D2) -- (A10);
\draw[doublehead] (A13) -- (F2);
\draw[doublehead] (F2) -- (D10);
\draw[doublehead] (D10) -- (Abot);

\pic ["$x$"] {angle = G13--F13--Btop};
\pic ["$y$"] {angle = F13--Btop--G13};
\pic ["$z$"] {angle = Btop--G13--F13};

\pic ["$x$"] {angle = Atop--Btop--F13};
\pic ["$y$"] {angle = Btop--F13--Atop};
\pic ["$z$"] {angle = F13--Atop--Btop};

\pic ["$x$"] {angle = F13--G13--D2};
\pic ["$y$"] {angle = G13--D2--F13};
\pic ["$z$"] {angle = D2--F13--G13};

\pic ["$x$"] {angle = Btop--Ctop--G13};
\pic ["$y$"] {angle = Ctop--G13--Btop};
\pic ["$z$"] {angle = G13--Btop--Ctop};

\pic ["$x$"] {angle = A13--G13--Ctop};
\pic ["$y$"] {angle = G13--Ctop--A13};
\pic ["$z$"] {angle = Ctop--A13--G13};

\pic ["$x$"] {angle = G13--A13--E2};
\pic ["$y$"] {angle = A13--E2--G13};
\pic ["$z$"] {angle = E2--G13--A13};

\pic ["$x$"] {angle = E2--D2--G13};
\pic ["$y$"] {angle = D2--G13--E2};
\pic ["$z$"] {angle = G13--E2--D2};

\pic ["$x$"] {angle = F2--E2--A13};
\pic ["$y$"] {angle = E2--A13--F2};
\pic ["$z$"] {angle = A13--F2--E2};

\pic ["$x$"] {angle = B10--A10--D2};
\pic ["$y$"] {angle = A10--D2--B10};
\pic ["$z$"] {angle = D2--B10--A10};

\pic ["$x$"] {angle = C10--B10--E2};
\pic ["$y$"] {angle = B10--E2--C10};
\pic ["$z$"] {angle = E2--C10--B10};

\pic ["$x$"] {angle = D10--C10--F2};
\pic ["$y$"] {angle = C10--F2--D10};
\pic ["$z$"] {angle = F2--D10--C10};

\pic ["$x$"] {angle = D2--E2--B10};
\pic ["$y$"] {angle = E2--B10--D2};
\pic ["$z$"] {angle = B10--D2--E2};

\pic ["$x$"] {angle = E2--F2--C10};
\pic ["$y$"] {angle = F2--C10--E2};
\pic ["$z$"] {angle = C10--E2--F2};

\pic ["$x$"] {angle = C10--D10--Abot};
\pic ["$y$"] {angle = D10--Abot--C10};
\pic ["$z$"] {angle = Abot--C10--D10};

\end{tikzpicture}  
  \caption{Labelling of the internal angles of $\mathbb{T}^2_s$.}\label{fig3}
\end{figure}
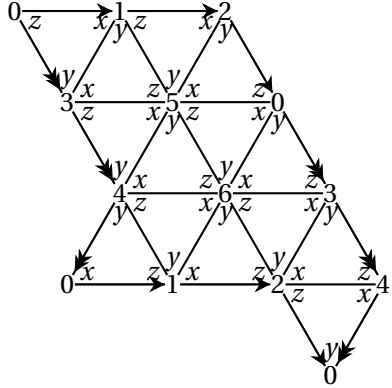
This labelling can be used to define a linear function on $\k L_2^s||N_2^s$ as follows: it sends an element of $L_2^s||N_2^s$ to $1$ if the corresponding angle has the label $a$, to $q$ if the corresponding angle has the label $b$, and to $q^2$ if the corresponding angle has the label $c$. This linear function is readily seen to vanishing on the image of the differential, hence $HH^2(I_q(\mathbb{T}^2_s),I_q(\mathbb{T}^2_s))\ne 0$, as required. An Euler characteristic computation immediately proves that in this case $\dim HH^1(I_q(\mathbb{T}^2_s),I_q(\mathbb{T}^2_s))=2$. The verification that the cup product is the same as that of the torus is a bit more tedious and is omitted. 
\end{proof}

Our result shows that, while each of the incidence algebras $I(\mathbb{T}^2_s)$ and $I(\mathbb{T}^2_c)$ is, in a sense, an algebraic model of a torus, their deformations capture some intricate differences between the simplicial and the cubical realizations. In the following section, we shall see an even more suprising phenomenon of a similar flavour.

\section{Two derived versions of \texorpdfstring{$\mathbb{P}^1\times\mathbb{P}^1$}{P1P1}}\label{sec:exceptional}

In this section, we examine two examples of nonzero cup products arising from two different derived versions for $\mathbb{P}^1\times\mathbb{P}^1$, by which we mean finite-dimen\-sional algebras $\Lambda$ for which we have an equivalence of derived categories
\begin{equation}\label{eq:derived}
D^b(\Mod(\Lambda))\simeq D^b(\text{coh}(\mathbb{P}^1\times \mathbb{P}^1)).   
\end{equation}
Work of Lowen and Van den Bergh on Hochschild cohomology of abelian categories \cite{MR2183254}, together with the Hochschild--Kostant--Rosenberg theorem for Hoch\-schild cohomology of smooth schemes \cite[Cor.~0.6]{MR1940241}, implies that for such algebra over a field $\k$ of zero characteristic, we have a Gerstenhaber algebra isomorphism 
 \[
HH^\bullet(\Lambda,\Lambda)\cong H^\bullet(\mathbb{P}^1\times \mathbb{P}^1, \wedge^\bullet \mathrm{T}_{\mathbb{P}^1\times \mathbb{P}^1})\cong  
H^\bullet(\mathbb{P}^1, \wedge^\bullet \mathrm{T}_{\mathbb{P}^1})^{\otimes 2}.  
 \] 
Since we have 
\begin{gather*}
H^\bullet(\mathbb{P}^1, \wedge^\bullet \mathrm{T}_{\mathbb{P}^1})\cong \k 1,\\
H^\bullet(\mathbb{P}^1, \wedge^\bullet \mathrm{T}_{\mathbb{P}^1})\cong\Lie(\Aut(\mathbb{P}^1)) \cong \mathfrak{sl}_2,
\end{gather*}
it follows that 
\begin{gather*}
H^0(\mathbb{P}^1\times \mathbb{P}^1, \wedge^\bullet \mathrm{T}_{\mathbb{P}^1\times \mathbb{P}^1})\cong \k 1,\\
H^1(\mathbb{P}^1\times \mathbb{P}^1, \wedge^\bullet \mathrm{T}_{\mathbb{P}^1\times \mathbb{P}^1})\cong \mathfrak{sl}_2\oplus \mathfrak{sl}_2,\\
H^2(\mathbb{P}^1\times \mathbb{P}^1, \wedge^\bullet \mathrm{T}_{\mathbb{P}^1\times \mathbb{P}^1})\cong \mathfrak{sl}_2\otimes \mathfrak{sl}_2,     
\end{gather*}
where these isomorphisms are compatible with the Lie bracket on $HH^1$, the adjoint action of $HH^1$ on $HH^2$, and the cup product: the cup product of two different copies of $\mathfrak{sl}_2$ in $HH^1$ is identified with $\mathfrak{sl}_2\otimes \mathfrak{sl}_2\cong HH^2$. In particular, the Hochschild cohomology of such an algebra always has a nonzero cup product. As alluded to in the introduction, general methods to produce derived equivalences like \eqref{eq:derived} give triangular algebras.

\subsection{Derived \texorpdfstring{$\mathbb{P}^1\times\mathbb{P}^1$}{P1P1} of Van den Bergh}\label{sec:vdB}

The example in this section is of geometric nature, and goes back to Van den Bergh \cite[Ex.~4.1]{MR2836401}, see also \cite[Prop.~19]{MR3988086}. 

Let $Q^{(4)}$ be the quiver 
 \[
\begin{tikzcd}
  1 & 2 & 3 & 4
  \arrow["{y_0}"', shift right, from=1-1, to=1-2]
  \arrow["{x_0}", shift left, from=1-1, to=1-2]
  \arrow["{y_1}"', shift right, from=1-2, to=1-3]
  \arrow["{x_1}", shift left, from=1-2, to=1-3]
  \arrow["{y_2}"', shift right, from=1-3, to=1-4]
  \arrow["{x_2}", shift left, from=1-3, to=1-4]
\end{tikzcd}
 \]
We consider the algebra $\Pi$ that is the quotient of the path algebra of $Q^{(4)}$ by the relations 
\begin{gather*}
y_2x_1x_0=x_2x_1y_0,\\
x_2y_1y_0=y_2y_1x_0.
\end{gather*}

It is known \cite[Prop.~19]{MR3988086} that the derived category of $\Pi$-modules is equivalent to $D^b(\text{coh}(\mathbb{P}^1\times \mathbb{P}^1))$. The discussion above immediately implies the following result. 

\begin{proposition}
We have an isomorphism of Gerstenhaber algebras
 \[
HH^\bullet(\Pi,\Pi)\cong H^\bullet(\mathbb{P}^1, \wedge^\bullet \mathrm{T}_{\mathbb{P}^1})^{\otimes 2}.     
 \]
In particular, the cup product of $HH^\bullet(\Pi,\Pi)$ is nonzero. 
\end{proposition}

\subsection{Product of two derived \texorpdfstring{$\mathbb{P}^1$}{P1}}\label{sec:prodP1}

Throughout this section, we denote by $\Lambda$ to be the path algebra of the Kronecker quiver $Q=\begin{tikzcd}
    1 & 2
    \arrow[shift right, from=1-1, to=1-2]
    \arrow[from=1-1, to=1-2]
\end{tikzcd}$. To avail of a geometric viewpoint from time to time, we shall assume that $\mathrm{char}(\k)=0$, though many of our results can be proved by purely algebraic methods under the weaker assumption $\mathrm{char}(\k)\ne 2$.

It is a foundational result \cite{MR509388} that the derived category of $\Lambda$-modules is equivalent to $D^b(\text{coh}(\mathbb{P}^1))$. A straightforward consequence of this result is that the derived category of $\Lambda\otimes\Lambda$-modules is equivalent to $D^b(\text{coh}(\mathbb{P}^1\times\mathbb{P}^1))$, implying the following result. 

\begin{proposition}\label{prop:GerstP1P1}
We have an isomorphism of Gerstenhaber algebras
 \[
HH^\bullet(\Lambda\otimes\Lambda,\Lambda\otimes\Lambda)\cong H^\bullet(\mathbb{P}^1, \wedge^\bullet \mathrm{T}_{\mathbb{P}^1})^{\otimes 2}.     
 \]
In particular, the cup product of $HH^\bullet(\Lambda\otimes\Lambda,\Lambda\otimes\Lambda)$ is nonzero. 
\end{proposition}

Let us remark that Green and Solberg \cite{MR2067380} indicate that tensor products of algebras may be used to provide examples of algebras with nonzero cup products of Hochschild cohomology, so the example of $\Lambda\otimes\Lambda$ is not at all new. However, it turns out that there is an interesting aspect of this example that does not seem to have been studied: deformations of the algebra $\Lambda\otimes\Lambda$. 

In the next section, we shall examine Hochschild cohomology of those deformations, and for that we need to fix some notations and identifications that will prove useful below. The quiver $Q\times Q$ of the algebra $\Lambda\otimes\Lambda$
has vertices $Q_0\times Q_0$, arrows $Q_1\times Q_0\sqcup Q_0\times Q_1$, and paths of length two 
$(Q_1\times e_2)(e_1\times Q_1)\sqcup (e_2\times Q_1)(Q_1\times e_1)$; there are no paths of length three. The relations of the algebra $\Lambda\otimes\Lambda$ are of the form
\begin{equation}\label{eq:rel-undef}
(\alpha\times e_2)(e_1\times \beta)=(e_2\times \beta)(\alpha\times e_1)   
\end{equation}
for all possible choices of $\alpha,\beta\in Q_1$. As we mentioned above, we have  
 \[
HH^1(\Lambda,\Lambda)\cong H^\bullet(\mathbb{P}^1, \wedge^\bullet \mathrm{T}_{\mathbb{P}^1})\cong \mathfrak{sl}_2,  
 \]
and it is easy to see from independent algebraic computations \cite{MR4102135} that under this identification $\mathfrak{sl}_2$ acts on $\k Q_1$ according to the standard two-dimensional representation $L(1)$. Thus, the relations \eqref{eq:rel-undef} are parametrized by $L(1)\otimes L(1)$; we shall designate the path $(\alpha\times e_2)(e_1\times \beta)$ to be the leading term of the relation, and the path $(e_2\times \beta)(\alpha\times e_1)$ to be the irreducible path of the relation. In particular, both the space of the reducible paths of length $2$ and the space of irreducible paths of length $2$ for the algebra $\Lambda\otimes\Lambda$ are identified with $L(1)\otimes L(1)$; for each $u\in L(1)\otimes L(1)$, we shall denote by $\overline{u}$ the corresponding reducible path and by $u$ the corresponding irreducible path, so that the space of relations of $\Lambda\otimes\Lambda$ under this identification consists of all elements $\overline{u}-u$, $u\in L(1)\otimes L(1)$. 

\begin{proposition}\label{prop:lldef}
The algebra $\Lambda\otimes\Lambda$ admits, up to equivalence, a nine-parametric family of deformations $(\Lambda\otimes\Lambda)_\Psi$ depending on a parameter $\Psi\in\mathfrak{sl}_2\otimes\mathfrak{sl}_2$; these deformations are defined over $\k[\mathfrak{sl}_2\otimes\mathfrak{sl}_2]\subset\k[[\mathfrak{sl}_2\otimes\mathfrak{sl}_2]]$. Under the above identification, the deformation $(\Lambda\otimes\Lambda)_\Psi$ has the space of relations
 \[
\overline{u}=u+\Psi(u),\quad u\in L(1)\otimes L(1),   
 \]
where the action of $\Psi\in\mathfrak{sl}_2\otimes\mathfrak{sl}_2$ on $L(1)\otimes L(1)$ is the Kronecker tensor product action.
\end{proposition}

\begin{proof}
First of all, it follows from Proposition \ref{prop:GerstP1P1} that 
 \[
HH^2(\Lambda\otimes\Lambda,\Lambda\otimes\Lambda)\cong\mathfrak{sl}_2\otimes\mathfrak{sl}_2,   
 \]
and therefore equivalence classes of infinitesimal deformations have $\mathfrak{sl}_2\otimes\mathfrak{sl}_2$ as the space of parameters. To describe these infinitesimal deformations, we repeat the strategy of Proposition \ref{prop:HHincid} and look at the Koszul complex of the algebra. It follows that the deformation corresponding to $\Psi\in\mathfrak{sl}_2\otimes\mathfrak{sl}_2$ adds to a relation of our algebra the image of the leading term of that relation under the action of $\Psi$. To perform this computation explicitly, it is convenient to take a purely algebraic viewpoint, and to use the known Gerstenhaber structrure on $HH^\bullet(\Lambda,\Lambda)$ \cite{MR4102135} and the result of Le and Zhou \cite{MR3175033}, who prove that for two associative algebras $\Lambda_1$ and $\Lambda_2$, we have a Gerstenhaber algebra isomorphism 
 \[
HH^\bullet(\Lambda_1\otimes \Lambda_2,\Lambda_1\otimes \Lambda_2)\cong HH^\bullet(\Lambda_1,\Lambda_1)\otimes HH^\bullet(\Lambda_2,\Lambda_2),     
 \] 
whenever at least one of the two algebras is finite-dimensional. Here the right hand side is equipped with the usual Gerstenhaber algebra structure coming from the Hopf operad structure of the Gerstenhaber operad \cite{MR2954392}. Since the identification
 \[
HH^1(\Lambda,\Lambda) \oplus  HH^1(\Lambda,\Lambda) \cong HH^1(\Lambda\otimes\Lambda,\Lambda\otimes\Lambda)     
 \]
sends $(X,Y)$ to $(X\otimes 1, 1\otimes Y)$ in  
 \[
HH^1(\Lambda,\Lambda) \otimes HH^0(\Lambda,\Lambda) \oplus  HH^0(\Lambda,\Lambda) \otimes HH^1(\Lambda,\Lambda)\cong HH^1(\Lambda\otimes\Lambda,\Lambda\otimes\Lambda),   
 \]
it follows that for $(X,Y)\in \mathfrak{sl}_2\oplus \mathfrak{sl}_2\cong  HH^1(\Lambda\otimes\Lambda,\Lambda\otimes\Lambda)$, and $a_1,a_2,b_1,b_2\in\Lambda$ we have $(X\cup Y)(a_1\otimes b_1,a_2\otimes b_2)=X(a_1)a_2 \otimes b_1Y(b_2)$.    
Since the leading terms of our relations are $(\alpha\times e_2)(e_1\times \beta)$, we have, for cup-decomposable products   
 \[
(X\cup Y)(\alpha\times e_2)(e_1\times \beta)=X(\alpha) \otimes Y(\beta),    
 \]
which is precisely the Kronecker tensor product action. 

To conclude, we also argue as in Proposition \ref{prop:HHincid}. The quiver  $Q\times Q$ has no path of length three, so our infinitesimal deformation is defined over $\k[\mathfrak{sl}_2\otimes\mathfrak{sl}_2]\subset\k[[\mathfrak{sl}_2\otimes\mathfrak{sl}_2]]$, and not just over $\k[\mathfrak{sl}_2\otimes\mathfrak{sl}_2]/(\mathfrak{sl}_2\otimes\mathfrak{sl}_2)^2$, as an infinitesimal deformation usually would be. Indeed, our relations take care of products of two arrows, and polynomials in $\mathfrak{sl}_2\otimes\mathfrak{sl}_2$ that vanish twice at the origin would only appear when computing products of at least three arrows. This completes the proof.  
\end{proof}

\subsection{Hochschild cohomology of deformations of derived \texorpdfstring{$\mathbb{P}^1\times\mathbb{P}^1$}{P1P1}}\label{sec:defP1P1}

Deformation questions closely related to the example of Section \ref{sec:vdB} were studied by Belmans in~\cite{MR3988086}; he in particular shows that all nontrivial deformations have smaller Hoch\-schild cohomology. We shall now discuss Hochschild cohomology of deformations $(\Lambda\otimes\Lambda)_\Psi$ of the example of Section \ref{sec:prodP1}. Some of these deformations are twisted tensor products and so the methods of \cite{MR4418319} apply; however, this only covers some particular cases, and we use a different approach.  

To understand how the answer depends on $\Psi\in \mathfrak{sl}_2\otimes\mathfrak{sl}_2$, we note that it clearly only depends on the orbit of $\Psi$ under the adjoint action of $SL_2\times SL_2$ (arising from the change of basis in the space of generators of each of the two factors $\Lambda$). If we use the suitably normalized Killing form 
 \[
\mathsf{k}\colon\mathfrak{sl}_2\times\mathfrak{sl}_2\to\k, \quad \mathsf{k}(a,b):=\tr(ab),  
 \]
to make an identification $\mathfrak{sl}_2^*\cong \mathfrak{sl}_2$, we may think of $\Psi$ as an element of $\End(\mathfrak{sl}_2)$, and the former action translates into the action of the group $SO(\mathfrak{sl}_2,\mathsf{k})\times SO(\mathfrak{sl}_2,\mathsf{k})$ on $\End(\mathfrak{sl}_2)$ given by $(P,Q)(\Psi)=P^\dagger\Psi Q$, where $P^\dagger$ means the adjoint of $P$ with respect to $\mathsf{k}$. Thus, the study of orbits is a version of the ``singular value decomposition'' \cite[Sec.~7.3]{MR1084815} for the Killing form. 

For a vector space $V$ equipped with a quadratic form $\mathsf{q}$, the usual singular value decomposition theorem obtains all the necessary information for classifying the orbits of $SO(V,\mathsf{q})\times SO(V,\mathsf{q})$ on linear transformations $\Psi\in\End(V)$ from the invariants of the conjugation action of $SO(V,\mathsf{q})$ on the self-adjoint transformations $\Psi^\dagger\Psi$. The following result shows the special role that the transformation $\Psi^\dagger\Psi$ plays in our case.

\begin{theorem}\label{th:psidef}
For each $\Psi\in\End(\mathfrak{sl}_2)$, we have an isomorphism of vector spaces 
 \[
HH^1((\Lambda\otimes\Lambda)_\Psi,(\Lambda\otimes\Lambda)_\Psi)\cong \mathrm{stab}(\Psi)\oplus \mathfrak{J}_\Psi,   
 \]
where $\mathrm{stab}(\Psi)$ is the Lie algebra of stabilizer of $\Psi$ in $SO(\mathfrak{sl}_2,\mathsf{k})\times SO(\mathfrak{sl}_2,\mathsf{k})$, and 
 \[
\mathfrak{J}_\Psi:=\{x\in \mathfrak{sl}_2\colon \Psi^\dagger\Psi(x)=4x\}.   
 \]
\end{theorem} 

The first Hochschild cohomology describes infinitesimal symmetries; in a sense, the first summand describes the obvious symmetries arising from the symmetries of the direction of deformation $\Psi$, while the second one corresponds to somewhat unexpected symmetries. Note, however, that the second summand is never present for $\Psi$ sufficiently ``small''; in a sense, this is a phenomenon that is completely analogous to the roots of unity appearing in Section \ref{sec:defS1S1}. 

\begin{proof}
We use the same strategy as in Proposition \ref{prop:HHincid}; namely, we shall compute the Hochschild cohomology $HH^\bullet((\Lambda\otimes\Lambda)_\Psi,(\Lambda\otimes\Lambda)_\Psi)$ via the Koszul complex. Under the identifications made above, that complex is 
 \[
0\to \k Q_0\times Q_0||Q_0\times Q_0 \to \k (Q_1\times Q_0\sqcup Q_0\times Q_1)||(Q_1\times Q_0\sqcup Q_0\times Q_1)\to \k L_2||N_2\to 0   , 
 \]
where $L_2$ is the set of leading terms of relations, and $N_2$ is the set of irreducible paths of length two. Examining the possible parallel paths of length $1$ makes it clear that we may further identify 
 \[
\k (Q_1\times Q_0\sqcup Q_0\times Q_1)||(Q_1\times Q_0\sqcup Q_0\times Q_1)\cong\End(L(1))\otimes \k Q_0\oplus \k Q_0\otimes \End(L(1)).    
 \]
Additionally, identifications made to state Proposition \ref{prop:lldef} suggest to identify 
 \[
\k L_2||N_2\cong \End(L(1)\otimes L(1)).   
 \]
Under all these identifications, we easily compute that the rightmost differential of that complex is given by
\begin{gather*}
\partial(f\otimes e_1)=-(1\otimes f)(1+\Psi),\\
\partial(f\otimes e_2)=(1+\Psi)(f\otimes 1),\\
\partial(e_1\otimes f)=(1+\Psi)(1\otimes f),\\
\partial(e_2\otimes f)=-(f\otimes 1)(1+\Psi),
\end{gather*}
which means that the kernel of $\partial$ consists of quadruples 
 $(f_1,f_2,f_3,f_4)\in \End(L(1))^4$ for which
\begin{equation}\label{eq:kernel}
(1+\Psi)(f_2\otimes 1+1\otimes f_3)=(f_4\otimes 1+1\otimes f_1)(1+\Psi).   
\end{equation}
Additionally, the image of the leftmost differential under this identification is spanned by the elements $1\otimes e_1+e_1\otimes 1$, $1\otimes e_2+e_2\otimes 1$, and $1\otimes e_1-e_2\otimes 1$, so modulo the image of the differential we may consider the same Equation \eqref{eq:kernel} for 
 \[
(f_1,f_2,f_3,f_4)\in \k (1,0,0,0)\oplus \mathfrak{sl}_2^{\oplus 4}\subset \End(L(1))^4.   
 \] 
Note that Equation \eqref{eq:kernel} is written in $\End(L(1)\otimes L(1))$ which decomposes under the action of $\mathfrak{sl}_2\oplus\mathfrak{sl_2}$ 
 \[
L(0)\otimes L(0)\oplus L(0)\otimes L(2)\oplus L(2)\otimes L(0)\oplus L(2)\otimes L(2).   
 \]
To split \eqref{eq:kernel} according to this decomposition, we shall use a simple result which is an easy consequence of the Cayley--Hamilton theorem: for every two matrices $a,b\in\mathfrak{sl}_2$, we have $ab+ba=\tr(ab)1=\mathsf{k}(a,b)1$. Using this fact, we note that if change variables and write 
 \[
(f_1,f_2,f_3,f_4)=c(1,0,0,0)+(u_1-v_1,u_2-v_2,u_1+v_1,u_2+v_2)   
 \]
with $u_1,u_2,v_1,v_2\in\mathfrak{sl}_2$, Equation \eqref{eq:kernel} is equivalent to the system of equations
 \[
\begin{cases}
c=0,\\
[u_2\otimes 1+1\otimes u_1,\Psi]=0,\\
v_2\dashv\Psi-2v_1=0,\\
\Psi\vdash v_1-2v_2=0.
\end{cases}   
 \] 
The first equation here is self-explanatory, the second one describes the stabilizer of $\Psi$ in $\mathfrak{sl}_2\oplus \mathfrak{sl}_2$, and the last two equations use the maps  
 \[
\dashv\colon \mathfrak{sl}_2\otimes(\mathfrak{sl}_2\otimes\mathfrak{sl}_2)\to\mathfrak{sl}_2\quad  \text{ and  }\quad 
\vdash\colon (\mathfrak{sl}_2\otimes\mathfrak{sl}_2)\otimes\mathfrak{sl}_2\to\mathfrak{sl}_2,
 \]
which compute the Killing form between the alone factor $\mathfrak{sl}_2$ and one of the two factors in $\mathfrak{sl}_2\otimes\mathfrak{sl}_2$.

From the last two equations of the above system one obtains 
 \[
\Psi\vdash (v_2\dashv\Psi)=\Psi\vdash(2v_1)=4v_2.   
 \]
Conversely, if $\Psi\vdash (v_2\dashv\Psi)=4v_2$, then an element $v_1$ satisfying the above system can be uniquely reconstructed as $v_1:=\frac12 v_2\dashv\Psi$. Finally, a direct inspection shows that under our identification $\mathfrak{sl}_2\otimes\mathfrak{sl}_2\cong \End(\mathfrak{sl}_2)$, we have 
 \[
 \Psi\vdash (v_2\dashv\Psi)=\Psi^\dagger\Psi(v_2),  
 \]
which completes the proof. 
\end{proof}

In principle, the description of $HH^1((\Lambda\otimes\Lambda)_\Psi,(\Lambda\otimes\Lambda)_\Psi)$ in Theorem \ref{th:psidef} relies on completely classical ingredients. If $\k$ is algebraically closed, the classification of self-adjoint transformations under orthogonal similarity is known; for instance, it is spelled out explicitly for $\k=\mathbb{C}$ in \cite[Chapter XI.3]{MR107649}; according to \cite[Chapter 8]{MR2964027}, this essentially goes back to Weiler \cite{zbMATH02718341}. Orbit stabilizers for that action are also completely understood \cite{MR4578016}. Using all that information, one may obtain all sorts of results about the complex singular value decomposition, see~\cite{MR900070} for generic orbits and \cite[Th.~13]{8603813} for the general case. To compute
 \[
\dim HH^1((\Lambda\otimes\Lambda)_\Psi,(\Lambda\otimes\Lambda)_\Psi)=\dim\mathrm{stab}(\Psi)+\dim\mathfrak{J}_\Psi,   
 \]
one may pass to the algebraic closure of $\k$ and avail of the results of \cite{MR4578016}. One can also use a more brute force computational approach. For that, we computed directly the following matrices whose ranks determine $\dim\mathrm{stab}(\Psi)$ and $\dim\mathfrak{J}_\Psi$. 

\begin{proposition}\label{prop:twomatrices}
Let us consider an element $\Psi\in\mathfrak{sl}_2\otimes\mathfrak{sl}_2$ of the form
 \[
x_1 e\otimes e+x_2 e\otimes h+x_3e\otimes f+x_4 h\otimes e+x_5 h\otimes h+x_6 h\otimes f+x_7 f\otimes e+x_8f\otimes h+x_9f\otimes f.    
 \]
Then $\mathrm{stab}(\Psi)$ may be identified with the kernel of the matrix 
 \[
U_\Psi= 
\begin{pmatrix}
-x_4&x_1&0&-x_2&x_1&0\\
-x_5&x_2&0&\frac12x_3&0&-\frac12x_1\\
-x_6&x_3&0&0&-x_3&x_2\\
\frac12x_7&0&-\frac12x_1&-x_5&x_4&0\\
\frac12x_8&0&-\frac12x_2&\frac12x_6&0&-\frac12x_4\\
\frac12x_9&0&-\frac12x_3&0&-x_6&x_5\\
0&-x_7&x_4&-x_8&x_7&0\\
0&-x_8&x_5&\frac12x_9&0&-\frac12x_7\\
0&-x_9&x_6&0&-x_9&x_8 
\end{pmatrix},
 \]
and $\mathfrak{J}_\Psi$ may be identified with the kernel of the matrix 
 \[
V_\Psi= 
\begin{pmatrix} 
x_1x_9 + x_3x_7 + 2x_4x_6-4&   2x_1x_8 + 2x_2x_7 + 4x_4x_5&   2x_1x_7 + 2x_4^2\\
x_2x_9 + x_3x_8 + 2x_5x_6&   4x_2x_8 + 4x_5^2-4&   x_1x_8 + x_2x_7 + 2x_4x_5\\
2x_3x_9 + 2x_6^2&   2x_2x_9 + 2x_3x_8 + 4x_5x_6&   x_1x_9 + x_3x_7 + 2x_4x_6-4   
\end{pmatrix}.
 \]
\end{proposition}

The ranks of these matrices can be studied by applying Gröbner bases to the corresponding determinantal ideals~\cite{MR2375719}; we did that using \texttt{Magma} \cite{MR1484478} to double check various claims made below. Let us remark that the first of those matrices is, up to minor corrections accounting to a different choice of basis, the matrix $A_\pi$ from \cite[Th.~3.6]{MR2802550} associated to a given holomorphic bivector field $\pi$ on $\mathbb{P}^1\times\mathbb{P}^1$; this matrix is instrumental in computing the Poisson cohomology of $(\mathbb{P}^1\times\mathbb{P}^1,\pi)$. Thus, the Poisson cohomology corresponds to the ``obvious infinitesimal symmetries'' alluded to after the statement of Theorem \ref{th:psidef}.\\

Note that from our description of the vector space $HH^1((\Lambda\otimes\Lambda)_\Psi,(\Lambda\otimes\Lambda)_\Psi)$, one can also immediately compute the dimension of the second Hochschild cohomology $HH^2((\Lambda\otimes\Lambda)_\Psi,(\Lambda\otimes\Lambda)_\Psi)$. Indeed, the Euler characteristic of the Koszul complex of $(\Lambda\otimes\Lambda)_\Psi$ is $4-16+16=4$, and that complex vanishes in degrees greater than two, which implies that  
 \[
\dim HH^2((\Lambda\otimes\Lambda)_\Psi,(\Lambda\otimes\Lambda)_\Psi)=\dim HH^1((\Lambda\otimes\Lambda)_\Psi,(\Lambda\otimes\Lambda)_\Psi)+3.
 \]
Moreover, not all possible combinations of dimensions are possible; at a first glance, the situation is similar to that summarized in \cite[Table 2]{MR3988086} for the example from Section \ref{sec:vdB}: the triple 
 \[
\dim HH^\bullet:=(\dim HH^0,\dim HH^1,\dim HH^2)   
 \]
could only have the values 
 \[
(1,0,3), (1,1,4), (1,2,5), (1,3,6), (1,6,9).   
 \]
Specifically, the following table with rows corresponding to possible values of $\dim(\mathrm{stab}(\Psi))$ and columns corresponding to possible values of $\dim\mathfrak{J}_\Psi$ displays all feasible combinations of those dimensions.
 \[
\begin{tabular}{|c|c|c|c|c|}
\hline
&3&2&1&0\\
\hline
6&-&-&-&\checkmark \\
\hline
5&-&-&-&-\\
\hline
4&-&-&-&-\\
\hline
3&\checkmark&-&-&\checkmark\\
\hline
2&-&-&\checkmark&\checkmark\\
\hline
1&-&\checkmark&\checkmark&\checkmark\\
\hline
0&-&-&\checkmark&\checkmark\\
\hline
\end{tabular} 
 \]
In particular, one finds that $\dim HH^1=0$ in the generic cases of semi-simple $\Psi^\dagger\Psi$ with distinct eigenvalues all different from $4$. However, at a closer look one observes that, while in \cite[Table 2]{MR3988086} only the trivial deformations have the same dimensions of the Hochschild cohomology as the undeformed algebra, in our case we also have $\dim HH^1=6$ in the situation where the multiplicity of the eigenvalue $4$ of $\Psi^\dagger\Psi$ is equal to three, that is $\Psi^\dagger\Psi=4$. This follows from an unexpected property of the matrices $U_\Psi$ and $V_\Psi$ from Proposition \ref{prop:twomatrices}: the values of $\Psi$ for which the matrix $U_\Psi$ has rank three are exactly the same for which the matrix $V_\Psi$ is the zero matrix; in fact, parameters for which this happens form a three-dimensional subvariety in the nine-dimensional space of parameters. 

It is indicated in \cite[Sec.~3.2]{MR3988086} that computations for the deformed algebras suggest that the cup product on the Hochschild cohomology is often zero on elements in positive degrees. We observed a similar phenomenon when we computed the cup product using \texttt{QPA} \cite{qpa} for some randomly chosen deformations. Here are some examples we looked at:
 \[
\begin{array}{|c|c|c|c|}
\hline
\Psi&\dim\mathrm{stab}(\Psi)&\dim\mathfrak{J}_\Psi&\dim HH^\bullet\\
\hline
2e\otimes e+2f\otimes f+h\otimes h&3&3&(1,6,9)\\
\hline
e\otimes e&3&0&(1,3,6)\\
\hline
(e+h)\otimes(e+h)&2&1&(1,3,6)\\
\hline
e\otimes e+h\otimes h+2e\otimes f+2f\otimes e&1&2&(1,3,6)\\
\hline
(e+h+f)\otimes(e+h+f)&2&0&(1,2,5)\\
\hline
e\otimes e+f\otimes f+h\otimes h&1&1&(1,2,5)\\
\hline
e\otimes e+f\otimes f&1&0&(1,1,4)\\
\hline
e\otimes e+h\otimes h+f\otimes f+2e\otimes f+2f\otimes e&0&1&(1,1,4)\\
\hline
e\otimes e+2h\otimes h+e\otimes f+f\otimes e&0&0&(1,0,3)\\
\hline
\end{array}
 \] 
In almost all of these cases, all cup products of elements of positive degree are equal to zero, with the exception of the case $\Psi=(e+h+f)\otimes(e+h+f)$, where the only \emph{a priori} nonzero cup product is indeed non-zero. In particular, for the deformation parameter $2e\otimes e+2f\otimes f+h\otimes h$, one observes a very surprising phenomenon: there is no drop in $\dim HH^\bullet$ after deformation, but the cup product is trivial. Overall, comparing our results with those of~\cite{MR3988086}, one sees that while two derived equivalent algebras always have isomorphic Hochschild cohomology \cite{MR1099084}, the cohomological behaviour of their deformations may differ in a very substantial way for special values of the deformation parameters. 

\printbibliography
\end{document}